\documentclass{amsart}

\usepackage[latin1]{inputenc}
\usepackage[T1]{fontenc} 
\usepackage[british]{babel}

\DeclareMathOperator{\Hom}{Hom}
\DeclareMathOperator{\res}{res}
\DeclareMathOperator{\tr}{tr}
\DeclareMathOperator{\IM}{Im}
\DeclareMathOperator{\RE}{Re}

\usepackage{amssymb}

\usepackage{amsthm}

\theoremstyle{plain}
\newtheorem{theorem}{Theorem}
\newtheorem{lemma}{Lemma}

\theoremstyle{definition}
\newtheorem{definition}{Definition}

\theoremstyle{remark}
\newtheorem{example}{Example}

\begin{document}

\title[Spectral expansions of random sections of homogeneous vector bundles]{Spectral expansions of random sections \\ of homogeneous vector bundles}

\author{A. Malyarenko}
\address{Division of Applied Mathematics, School of Education, Culture and Communication, M\"{a}\-lar\-da\-len University, Box 883, SE-721 23 V\"{a}ster{\aa}s, Sweden}
\thanks{This paper is written with the financial support of the Data Innovation Research Institute, Cardiff University, UK}
\date{\today}
\dedicatory{Dedicated to my teacher Mykhailo Yadrenko in occasion of his 85th birthday.}

\begin{abstract}
Tiny fluctuations of the Cosmic Microwave Background as well as various observable quantities obtained by spin raising and spin lowering of the effective gravitational lensing potential of distant galaxies and galaxy clusters, are described mathematically as isotropic random sections of homogeneous spin and tensor bundles. We consider the three existing approaches to rigourous constructing of the above objects, emphasising an approach based on the theory of induced group representations. Both orthogonal and unitary representations are treated in a unified manner. Several examples from astrophysics are included.
\end{abstract}

\maketitle

\section{Introduction}

At the shift of millennia, cosmology came to an era of high-precision measurements. Various cosmological quantities like the temperature and polarisation of the Cosmic Microwave Background, the gravitational shear and others, are described mathematically as single realisations of random fields. Their tiny random fluctuations are now observable.

A rigourous mathematical theory of scalar-valued random fields is developing from the 1920th, see \cite{MR2977490} and the references herein. However, cosmological applications require to introduce another type of random fields, the \emph{random sections of vector-, tensor- and spin bundles}.

The first mathematical model of a complex-valued spin random field on the sphere was constructed by Geller and Marinucci in \cite{MR2737761}. They formulated the classical theory of Newman and Penrose \cite{MR0194172} in the language of complex line bundles over the two-dimensional sphere. See also \cite{MR2870527} for the case of \emph{tensor-valued} random fields on the sphere.

Another version of such a model was proposed by the author, see \cite{MR2884225}. He used the fact that the Hilbert space $H$ of square-integrable sections of a homogeneous vector- or tensor bundle with the fiber $E$ carries the so called \emph{induced representation} of the symmetry group $G$ of the bundle's base. The powerful theory of induced representations can then be used for spectral analysis of random sections of the bundle.

Baldi and Rossi \cite{MR3170229} used the following observation. There is a map called the \emph{pullback} acting from the space $H$ to the space of $E$-valued square-integrable functions on $G$ satisfying some symmetry condition. Applying the pullback to a random section of the bundle, we obtain a more simple object, an $E$-valued random field on $G$. No information is loosing under the pullback, but the theory becomes easier.

To show the issues one can meet, consider a toy example. Let $M=S^1$ be the centred unit circle embedded into the plane $\mathbb{R}^2$. Let $X(\mathbf{x})$ be a complex-valued random field on $S^1$, that is, there is a probability space $(\Omega,\mathfrak{F},\mathsf{P})$ and a function $X(\mathbf{x},\omega)\colon S^1\times\Omega\to\mathbb{C}$ such that for any fixed $\mathbf{x}_0\in S^1$ the function $X(\mathbf{x}_0,\omega)\colon\Omega\to\mathbb{C}$ is a random variable. Assume a little bit more: the function $X(\mathbf{x},\omega)$ is measurable as a function of two variables, and has a finite variance: $\mathsf{E}[|X(\mathbf{x})|^2]<\infty$ for all $\mathbf{x}\in S^1$. In other words, $X(\mathbf{x})\in L^2(\Omega)$, where $L^2(\Omega)$ is the Hilbert space of all complex-valued random variables on $\Omega$ with inner product $(X,Y)=\mathsf{E}[X\overline{Y}]$. By the result of Marinucci and Peccati \cite{MR3064996}, the field $X(\mathbf{x})$ is mean-square continuous, that is, the map $X(\mathbf{x})\colon S^1\to L^2(\Omega)$ is continuous. Put $U=S^1\setminus\{(1,0)\}$. An element $\mathbf{x}\in U$ has the form $\mathbf{x}=(\cos\varphi,\sin\varphi)^{\top}$ with $\varphi\in(0,2\pi)$. Cover the set $U$ by the chart $\sigma\colon U\to(0,2\pi)$ acting by $\sigma(\mathbf{x})=\varphi$. The mean-square continuous random field $X(\mathbf{x})$ is completely determined by its restriction $X(\varphi)$ to the dense set $U$. In what follows, we use both notations.

Let $G=\mathrm{SO}(2)$ be the group of orthogonal $2\times 2$ matrices with unit determinant. An element
\[
g_{\psi}=
\begin{pmatrix}
  \cos\psi & \sin\psi \\
  -\sin\psi & \cos\psi
\end{pmatrix}
\]
of the group~$G$ acts on an element $\mathbf{x}\in S^1$ by matrix-vector multiplication. In other words, $G$ acts on $S^1$ by rotations.

\begin{definition}\label{def:1}
A random field $X(\mathbf{x})$ is called \emph{strictly isotropic} if its finite-dimensional distributions are invariant with respect to the above action, that is, for any positive integer $n$, for any $n$ distinct points $\mathbf{x}_1$, \dots, $\mathbf{x}_n\in S^1$, and for any $g\in G$, the random vectors $(X(\mathbf{x}_1),\dots,X(\mathbf{x}_n))^{\top}$ and $(X(g\mathbf{x}_1),\dots,X(g\mathbf{x}_n))^{\top}$ have the same distribution.
\end{definition}

Let $L^2(X)$ be the intersection of all closed subspaces of the space $L^2(\Omega)$ that contain the set $\{\,X(\mathbf{x})\colon\mathbf{x}\in S^1\,\}$. The group $G$ acts on the vector $X(\mathbf{x})$ by
\[
(gX)(\mathbf{x})=X(g^{-1}\mathbf{x}),\qquad g\in G.
\]
This action can be extended by linearity and continuity to the action of $G$ in $L^2(X)$. It is easy to see the following: if the random field $X(\mathbf{x})$ is isotropic, then the above action is a unitary representation of the group~$G$, that is, a continuous homomorphism of $G$ to the group of unitary operators in $L^2(X)$ equipped with the strong operator topology.

In Section~\ref{sec:jump} we give a quick introduction to representation theory for a reader without necessary background, using \cite{MR1410059} as a main reference. Our main results, the spectral expansions of invariant random fields in homogeneous real and complex vector bundles, are presented in Section~\ref{sec:3}. Several examples from astrophysics are presented in Section~\ref{sec:4}. In particular, Example~\ref{ex:3} shows the equivalence of the spectral expansions of the polarised Cosmic Microwave Background given by Kamionkowski et al in \cite{Kamionkowski:1996ks} and by Zaldarriaga and Seljak in \cite{PhysRevD.55.1830}. In Example~4, we present an approach to the spectral expansions of the distortion random fields alternative both to expansions in spherical Bessel functions \cite{Castro_2005,Heavens_2003,Kitching_2011,Munshi_2011} and to wavelets expansions \cite{Leistedt_2012,Leistedt_2015}.

\section{A quick jump into representation theory}\label{sec:jump}

Let $G$ be a group with identity $e$, and $V$ be a set. A \emph{left action} of $G$ on $V$ is a map
\begin{equation}\label{eq:1}
\rho\colon G\times V\to V,\qquad (g,v)\mapsto \rho(g,v)=g\cdot v
\end{equation}
such that $e\cdot v=v$ and $(gh)\cdot v=g\cdot(h\cdot v)$.

Let $\mathbb{K}$ be either the field $\mathbb{R}$ of reals of the field $\mathbb{C}$ of complex numbers. A left action is called a \emph{representation} if $V$ is a linear space over $\mathbb{K}$ and all the left translations $\mathbf{v}\mapsto g\cdot\mathbf{v}$ are $\mathbb{K}$-linear. If $V$ is finite-dimensional, then we suppose in addition that the action \eqref{eq:1} is \emph{continuous}. In what follows we denote the representation \eqref{eq:1} by just $V$ if the action $\rho$ is thought.

A subspace $U\subset V$ is called \emph{invariant} if $g\cdot\mathbf{u}\in U$ for $g\in G$ and $\mathbf{u}\in U$. A representation $V\neq\{0\}$ is called \emph{irreducible} if it has no invariant subspaces other than $\{\mathbf{0}\}$ and $V$.

Denote by $V_s$ the subspace of $V$ generated by all finite-dimensional invariant subspaces. The subspace $V_s$ is invariant, since each element of $V_s$ is contained in a finite-dimensional invariant subspace.

Let $V$ and $W$ be two representations of a group~$G$. A linear map $f\colon V\to W$ is called an \emph{$G$-intertwining operator} if $f(g\cdot\mathbf{v})=g\cdot f(\mathbf{v})$ for $g\in G$ and $\mathbf{v}\in V$. Denote by $\Hom_G(V,W)$ the $\mathbb{K}$-linear space of all $G$-intertwining operators. We can make $G$ act on the above space by the representation
\begin{equation}\label{eq:2}
(g\cdot f)(\mathbf{v})=g\cdot(f(g^{-1}\cdot\mathbf{v})).
\end{equation}
The representations $V$ and $W$ are called \emph{equivalent} if the space $\Hom_G(V,W)$ contains an invertible operator.

In what follows we suppose that $G$ is a \emph{compact} topological group. There are finitely or countably many equivalence classes of irreducible representations of $G$. Moreover, each irreducible representation $V$ is finite-dimensional, and has an inner product $V\times V\to\mathbb{K}$, $(\mathbf{u},\mathbf{v})\mapsto\langle\mathbf{u},\mathbf{v}\rangle$ which is \emph{$G$-invariant}, that is, $\langle g\cdot\mathbf{u},g\cdot\mathbf{v}\rangle=\langle\mathbf{u},\mathbf{v}\rangle$ for all $g\in G$ and $\mathbf{u}$, $\mathbf{v}\in V$. Such a representation is called \emph{orthogonal} when $\mathbb{K}=\mathbb{R}$ and \emph{unitary} when $\mathbb{K}=\mathbb{C}$.

It remains to investigate the structure of a particular representation of a group $G$. Assume that $G$ is a \emph{Lie group}, that is: $G$ is a smooth manifold, and the group multiplication and taking the inverse element are smooth maps. Let $K$ be a closed subgroup of $G$, and let $(E,\langle\boldsymbol{\cdot},\boldsymbol{\cdot}\rangle)$ be a finite-dimensional representation of $K$ with $K$-invariant inner product. Let $G\times_KE$ be the quotient space of $G\times E$ under the equivalence relation $(g,\mathbf{v})\sim(gk,k^{-1}\cdot\mathbf{v})$ for $k\in K$. The map $(g,\mathbf{v})\mapsto gK$ respects the introduced relation, that is, equivalent points have the same image. Thus, the above map induces a map $p\colon G\times_KE\to G/K$. The triple $(G\times_KE,p,G/K)$ is a \emph{homogeneous vector bundle} with fiber $p^{-1}(x)$ isomorphic to $E$ for every $x\in G/K$. Traditionally, some of such bundles are called \emph{spin bundles}, see Example~\ref{ex:2} below, or \emph{tensor bundles} when the elements of $E$ are tensors, see Example~\ref{ex:3}.

A map $s\colon G/K\to G\times_KE$ is called a \emph{section} of the vector bundle $(G\times_KE,p,G/K)$, if $p\circ s$ is the identity map in $G/K$. Let $\mu$ be the probabilistic $G$-invariant Borel measure on $G/K$. Denote by $H$ the Hilbert space of the measurable sections satisfying
\[
\int_{G/K}\|s(x)\|^2_E\,\mathrm{d}\mu(x)<\infty
\]
with inner product
\[
\langle s_1,s_2\rangle=\int_{G/K}\langle s_1(x),s_2(x)\rangle_E\,\mathrm{d}\mu(x).
\]
Make $G$ act on $H$ by the representation
\[
(g\cdot s)(x)=gs(g^{-1}x).
\]
The representation $H_s$ is called the representation \emph{induced by the representation $E$ of the group~$K$}.

Let $V$ be an irreducible representation of the group~$G$. \emph{Schur's lemma} \cite[Chapter~I, Theorem~1.10]{MR1410059} says that every nonzero $f\in D(V)=\Hom_G(V,V)$ is invertible. Hence $D(V)$ is a \emph{finite-dimensional division algebra} over $\mathbb{K}$. When $\mathbb{K}=\mathbb{C}$, there is only one possibility, $\mathbb{C}$, see \cite[Theorem~7.6]{MR780184}. When $\mathbb{K}=\mathbb{R}$, there are only three possibilities, namely, $\mathbb{R}$, $\mathbb{C}$, and $\mathbb{H}$, the skew field of quaternions (Frobenius Theorem, \cite[p.~452]{MR780184}).

The \emph{composition} $f\circ\alpha$, $f\in\Hom_G(V,H_s)$, $\alpha\in D(V)$, turns $\Hom_G(V,H_s)$ into a $D(V)$-linear space. If $D(V)=\mathbb{H}$, then this space is a \emph{right} one. The \emph{evaluation} $\alpha(\mathbf{v})$, $\mathbf{v}\in V$, turns $V$ into another $D(V)$-linear space. If $D(V)=\mathbb{H}$, then this space is a \emph{left} one. Make $G$ act on the correctly defined tensor product $\Hom_G(V,H_s)\otimes_{D(V)}V$ by the representation
\[
g\cdot(f\otimes\mathbf{v})=(g\cdot f)\otimes(g\cdot\mathbf{v}),
\]
where the representation $g\cdot f$ is given by \eqref{eq:2}. Define the \emph{evaluation map}
\[
c_V\colon\Hom_G(V,H_s)\otimes_{D(V)}V\to H
\]
by
\[
c_V(f\otimes\mathbf{v})=f(\mathbf{v}).
\]

Denote by $\hat{G}$ the set of equivalence classes of irreducible representations of the group~$G$ and by $\mathbf{c}$ the algebraic direct sum of the linear maps $c_V$ over $V\in\hat{G}$. By \cite[Chapter~III, Proposition~1.7]{MR1410059}, the map
\[
\mathbf{c}\colon\sum_{V\in\hat{G}}\oplus\Hom_G(V,H_s)\otimes_{D(V)}V\to H_s
\]
is an invertible $G$-intertwining operator. In other words, the representation $H_s$ is equivalent to the direct sum of the representations $\Hom_G(V,H_s)\otimes_{D(V)}V$ over $V\in\hat{G}$. The image $c_V(\Hom_G(V,H_s)\otimes_{D(V)}V)$ is called the \emph{$V$-isotypical part} of $H_s$ and is the maximal subspace of $H_s$ whose irreducible subspaces are all equivalent to $V$.

We would like to calculate $n(V,H_s)$, the number of copies of $V$ containing in $H_s$. We have
\[
\begin{aligned}
n(V,H_s)&=\frac{\dim_{\mathbb{K}}c_V(\Hom_G(V,H_s)\otimes_{D(V)}V)}{\dim_{\mathbb{K}}V}\\
&=\frac{\dim_{\mathbb{K}}(\Hom_G(V,H_s)\otimes_{D(V)}V)}{\dim_{\mathbb{K}}V},
\end{aligned}
\]
because $c_V$ is invertible. Furthermore,
\[
\begin{aligned}
n(V,H_s)&=\frac{\dim_{D(V)}(\Hom_G(V,H_s)\otimes_{D(V)}V)\dim_{\mathbb{K}}D(V)}{\dim_{\mathbb{K}}V}\\
&=\dim_{D(V)}\Hom_G(V,H_s)\dim_{D(V)}V\frac{\dim_{\mathbb{K}}D(V)}{\dim_{\mathbb{K}}V}\\
&=\frac{\dim_{\mathbb{K}}\Hom_G(V,H_s)}{\dim_{\mathbb{K}}D(V)}\cdot
\frac{\dim_{\mathbb{K}}V}{\dim_{\mathbb{K}}D(V)}\cdot\frac{\dim_{\mathbb{K}}D(V)}{\dim_{\mathbb{K}}V}\\
&=\frac{\dim_{\mathbb{K}}\Hom_G(V,H_s)}{\dim_{\mathbb{K}}D(V)}.
\end{aligned}
\]

In order to calculate $\dim_{\mathbb{K}}\Hom_G(V,H_s)$, we use \emph{Frobenius reciprocity} \cite[Chapter~III, Proposition~6.2]{MR1410059}. It says that the representation $\Hom_G(V,H_s)$ is equivalent to the representation $\Hom_K(\res^G_KV,E)$, where $\res^G_KV$ is the restriction of the representation $V$ to $K$. In particular, we have
\begin{equation}\label{eq:3}
\dim_{\mathbb{K}}\Hom_G(V,H_s)=\dim_{\mathbb{K}}\Hom_K(\res^G_KV,E).
\end{equation}
Assume \emph{$E$ is irreducible}. The right hand side of Equation \eqref{eq:3} is $n(E,\res^G_KV)$, and we obtain
\begin{equation}\label{eq:4}
n(V,H_s)=\frac{n(E,\res^G_KV)}{\dim_{\mathbb{K}}D(V)}.
\end{equation}
In particular, when $\mathbb{K}=\mathbb{C}$, we have $\dim_{\mathbb{C}}D(V)=\dim_{\mathbb{C}}\mathbb{C}=1$, and we recover the classical form of Frobenius reciprocity: \emph{the multiplicity of $V$ in $H_s$ is equal to the multiplicity of $E$ in $\res^G_KV$}.

When $E$ is not irreducible, we use the following result: the representation induced by a direct sum is equivalent to the direct sum of representations induced by the components of the direct sum. Assume that $E$ contains $n_1$ copies of the irreducible representation $E_1$, \dots, $n_p$ copies of the irreducible representation $E_p$. Then we have
\[
n(V,H_s)=\frac{1}{\dim_{\mathbb{K}}D(V)}\sum_{i=1}^{p}n_in(E_i,\res^G_KV).
\]

The \emph{generalised theorem of Peter and Weyl} \cite[Chapter~III, Theorem~5.7]{MR1410059} says that $H_s$ is dense in $H$. Therefore, the representation $H$ has the same structure as $H_s$ has.

\section{Invariant random fields in vector bundles}\label{sec:3}

For any $x\in G/K$, let $\mathbf{X}(x)$ be a random vector in the fiber $p^{-1}(x)$ isomorphic to $E$. Call $\mathbf{X}(x)$ a \emph{random field in the vector bundle $(G\times_KE,p,G/K)$} if, in addition, the function $\mathbf{X}(x,\omega)$ is measurable as a function of two variables. Definition~\ref{def:1} can be modified as follows.

\begin{definition}
The random field $\mathbf{X}(x)$ is called \emph{strictly invariant} if its finite-dimensional distributions are invariant with respect to the action $x\mapsto g^{-1}x$ of the group~$G$.
\end{definition}

Assume that $\mathsf{E}[\|\mathbf{X}(x)\|^2_E]<\infty$ and let $L^2(\mathbf{X})$ be the closed linear span of the vectors $\mathbf{X}(x)$, $x\in G/K$, in the space $L^2(\Omega,E)$. Assume that $\mathbf{X}(x)$ is \emph{mean-square continuous}, that is, the map $\mathbf{X}\colon G/K\to L^2(\Omega)$ is continuous. Then for any $g\in G$ the domain of the map defined by
\[
U(g)\mathbf{X}(x)=\mathbf{X}(g^{-1}x)
\]
can be extended by linearity and continuity to the space $L^2(X)$. We denote the extended map by the same symbol, $U(g)$, and observe that $U(g)$ is a continuous representation of $G$ in $L^2(\mathbf{X})$ and the inner product of the above space is $G$-invariant. We investigate the structure of the above representation.

Assume first that $E$ is irreducible with an orthonormal basis $\{\,\mathbf{e}_j\colon 1\leq j\leq\dim_{\mathbb{K}}E\,\}$. Denote by $\hat{G}_K(E)$ the set of all $V\in\hat{G}$ with $n(E,\res^G_KV)>0$ and by $V_0$ the trivial representation of the group $G$: $V_0(g)=1$ for all $g\in G$. For each $V\in\hat{G}_K(E)$, fix an orthonormal basis ${}_E\mathbf{Y}^m_{Vn}(x)$ in the space of the $n$th copy of the representation $V$, where $1\leq m\leq\dim_{\mathbb{K}}V$, $1\leq n\leq n(V,H_s)$, and $n(V,H_s)$ is defined by \eqref{eq:4}.

\begin{theorem}\label{th:1}
The random field $\mathbf{X}(x)$ has the form
\[
X_j(x)=\sum_{V\in\hat{G}_K(E)}\sum_{m=1}^{\dim_{\mathbb{K}}V}\sum_{n=1}^{n(V,H_s)}
\sum_{k=1}^{\dim_{\mathbb{K}}E}a_{Vnmk}({}_EY^m_{Vn})_j(x),
\]
where
\[
a_{Vnmk}=\int_{G/K}X_k(x)\overline{({}_EY^m_{Vn})_k(x)}\,\mathrm{d}\mu(x).
\]
If $V\neq V_0$, then $\mathsf{E}[a_{Vnmk}]=0$. Finally,
\[
\mathsf{E}[a_{Vnmk}\overline{a_{V'n'm'k'}}]=\delta_{VV'}\delta_{mm'}
R^{Vn}_{kk'}
\]
with
\[
\sum_{V\in\hat{G}_K(E)}\sum_{n=1}^{n(V,H_s)}\dim_{\mathbb{K}}V\tr[R^{Vn}]<\infty.
\]
\end{theorem}

\begin{proof}
When $\mathbb{K}=\mathbb{C}$ and $n(V,H_s)=1$, this is \cite[Theorem~2]{MR2884225}. Otherwise, this is a straightforward extension with one more summation over the index $n$.
\end{proof}

Let $H_{\mathbf{X}}$ be the subspace of $H$ spanned by the vectors $\{\,{}_E\mathbf{Y}^m_{Vn}(x)\colon R^{Vn}\neq 0\,\}$. Denote $\mathbf{a}_{Vnm}=(a_{Vnm1},\dots,a_{Vnm\dim_{\mathbb{K}}E})^{\top}\in L^2(\mathbf{X})$. Let $R^{Vn}=A^{Vn}(A^{Vn})^{\top}$ be the Cholesky decomposition of the matrix $R^{Vn}$. The map
\[
T\mathbf{a}_{Vnm}=A^{Vn}{}_E\mathbf{Y}^m_{Vn}(x)
\]
may be extended by linearity and continuity to an intertwining operator from $L^2(\mathbf{X})$ to $H_{\mathbf{X}}$. Thus, the representation $U(g)$ contains as many copies of the irreducible representation $V$ as many elements are containing in the set $\{\,n\colon 1\leq n\leq n(V,H_s),R^{Vn}\neq 0\,\}$.

If $E$ is not irreducible, suppose that $E$ is the direct sum of the irreducible representations $E_1$, \dots, $E_N$. Denote the components of an invariant random field $\mathbf{X}(x)$ by $X_j^{(i)}(x)$, $1\leq i\leq N$, $1\leq j\leq\dim_{\mathbb{K}}E_i$. Let $P_i$ be the orthogonal projection from $E$ to $E_i$. Let $H^i_s$ be the representation induced by $E_i$. The next result is a straightforward extension of \cite[Theorem~2]{MR2884225}.

\begin{theorem}
The random field $\mathbf{X}(x)$ has the form
\begin{equation}\label{eq:5}
X^{(i)}_j(x)=\sum_{V\in\hat{G}_K(E_i)}\sum_{m=1}^{\dim_{\mathbb{K}}V}\sum_{n=1}^{n(V,H^i_s)}
\sum_{k=1}^{\dim_{\mathbb{K}}E_i}a_{Vnmik}({}_{E_i}Y^m_{Vn})_j(x),
\end{equation}
where
\[
a_{Vnmik}=\int_{G/K}X^{(i)}_k(x)\overline{({}_{E_i}Y^m_{Vn})_k(x)}\,\mathrm{d}\mu(x).
\]
If $V\neq V_0$, then $\mathsf{E}[a_{Vnmik}]=0$. Finally,
\[
\mathsf{E}[a_{Vnmik}\overline{a_{V'n'm'i'k'}}]=\delta_{VV'}\delta_{mm'}
R^{Vn}_{ik,i'k'}
\]
with
\[
\sum_{i=1}^{N}\sum_{V\in\hat{G}_K(E_i)}\sum_{n=1}^{n(V,H^i_s)}\dim_{\mathbb{K}}V\tr[P_iR^{Vn}P_i]<\infty.
\]
\end{theorem}

\section{Examples}\label{sec:4}

\begin{example}[The absolute temperature of the Cosmic Microwave Background]

Put $\mathbb{K}=\mathbb{C}$, $G=\mathrm{SO}(3)$, $K=\mathrm{SO}(2)$, and let $E$ be the trivial representation of $K$. The irreducible unitary representations of $G$ are enumerated by nonnegative integers \cite{MR1410059}, denote them by $V^{\ell}$. The restriction of $V^{\ell}$ to $K$ is isomorphic to the direct sum of the irreducible unitary representations $\varphi\mapsto\mathrm{e}^{\mathrm{i}m\varphi}$, $-\ell\leq m\leq\ell$. In particular, the multiplicity of $E$ in $V^{\ell}$ is equal to $1$. It follows that the representation $H$ contains one copy of each representation $V^{\ell}$. The space $G/K$ is the unit sphere $S^2$. Denote by $H^{\ell}$ the subspace of the space $L^2(S^2)$ where the above copy acts. The space $H^{\ell}$ is spanned by \emph{spherical harmonics}, $Y_{\ell m}(\theta,\varphi)$, where $(\theta,\varphi)$ is the chart of the manifold $S^2$ known as \emph{spherical coordinates}. In what follows, we will also denote spherical harmonics by $Y_{\ell m}(\mathbf{n})$, where $\mathbf{n}\in\mathbb{R}^3$ with $\|\mathbf{n}\|=1$. By Theorem~\ref{th:1}, an invariant random field in $(\mathrm{SO}(3)\times_{\mathrm{SO}(2)}\mathbb{C}^1,p,S^2)$ has the form
\begin{equation}\label{eq:6}
T(\theta,\varphi)=\sum_{\ell=0}^{\infty}\sum_{m=-\ell}^{\ell}a_{\ell m}
Y_{\ell m}(\theta,\varphi),
\end{equation}
where
\[
a_{\ell m}=\int_{S^2}T(\theta,\varphi)\overline{Y_{\ell m}(\theta,\varphi)}\,
\mathrm{d}\mu(\theta,\varphi),
\]
and $\mathrm{d}\mu(\theta,\varphi)=\sin\theta\,\mathrm{d}\theta\,\mathrm{d}\varphi=\mathrm{d}\mathbf{n}$ is the $\mathrm{SO}(3)$-invariant measure on $S^2$ with $\mu(S^2)=4\pi$. $G$-invariant random fields with $G=\mathrm{SO}(3)$ or $G=\mathrm{O}(3)$ are usually called \emph{isotropic}. By Theorem~\ref{th:1} $\mathsf{E}[a_{\ell m}]=0$ if $\ell\neq 0$, $\mathsf{E}[a_{\ell m}\overline{a_{\ell'm'}}]=C_{\ell}\delta_{\ell\ell'}\delta_{mm'}$ and
\[
\sum_{\ell=0}^{\infty}(2\ell+1)C_{\ell}<\infty.
\]

The spherical harmonics satisfy the relation $Y_{\ell\,-m}(\mathbf{n})=(-1)^m\overline{Y_{\ell m}(\mathbf{n})}$. It follows that the random field~\eqref{eq:6} is real-valued if and only if
\begin{equation}\label{eq:7}
a_{\ell\,-m}=(-1)^m\overline{a_{\ell m}}.
\end{equation}
In this case, we can write the random field \eqref{eq:6} in the form
\begin{equation}\label{eq:8}
T(\mathbf{n})=\sum_{\ell=0}^{\infty}\sum_{m=-\ell}^{\ell}\tilde{a}_{\ell m}S_{\ell}^m(\mathbf{n}),
\end{equation}
where
\begin{equation}\label{eq:11}
\begin{aligned}
\tilde{a}_{\ell m}&=
\begin{cases}
  \sqrt{2}\IM a_{\ell m}, & \mbox{if } m<0, \\
  a_{\ell m}, & \mbox{if } m=0, \\
  \sqrt{2}\RE a_{\ell m}, & \mbox{if } m>0,
\end{cases} \\
S^m_{\ell}(\mathbf{n})&=
\begin{cases}
  \frac{\mathrm{i}}{\sqrt{2}}(Y_{\ell m}(\mathbf{n})-(-1)^mY_{\ell\,-m}(\mathbf{n})), & \mbox{if } m<0, \\
  Y_{\ell m}(\mathbf{n}), & \mbox{if } m=0, \\
  \frac{1}{\sqrt{2}}(Y_{\ell m}(\mathbf{n})+(-1)^mY_{\ell\,-m}(\mathbf{n})), & \mbox{if } m>0.
\end{cases}
\end{aligned}
\end{equation}
The functions $S^m_{\ell}(\mathbf{n})$ are \emph{real-valued spherical harmonics}. In contrast to their complex-valued counterparts, that are proportional to the matrix entries of irreducible \emph{unitary} representations of the group $\mathrm{SO}(3)$, the functions $S^m_{\ell}(\theta,\varphi)$ are proportional to the matrix entries of irreducible \emph{orthogonal} representations of the above group. Either Equations \eqref{eq:6} and \eqref{eq:7} or Equation \eqref{eq:8} describe the absolute temperature of the Cosmic Microwave Background. The set $\{\,C_{\ell}\colon\ell\geq 0\,\}$ is called the \emph{power spectrum} and carries a lot of cosmological information.

\end{example}

\begin{example}[Spin bundles]\label{ex:2}

Put $\mathbb{K}=\mathbb{C}$, $G=\mathrm{SO}(3)$, $K=\mathrm{SO}(2)$, $E_s=\mathbb{C}^1$ with the action of $K$ given by $\varphi\mapsto\mathrm{e}^{-\mathrm{i}s\varphi}$, $s\in\mathbb{Z}$. The restriction of the irreducible unitary representation $V^{\ell}$ to $K$ contains one copy of the representation $E_s$ if and only if $\ell\in\{|s|,|s|+1,\dots\}$. By Frobenius reciprocity,
\[
H_s=\sum_{\ell=|s|}^{\infty}\oplus V^{\ell}.
\]
The basis vectors of the subspace of the Hilbert space $L^2(S^2,\mathrm{SO}(3)\times_{\mathrm{SO}(2)}E_s)$
where the representation $W_{\ell}$ acts, are called \emph{spin~$s$ spherical harmonics} and are denoted by ${}_sY_{\ell m}(\mathbf{n})$. In particular, ${}_0Y_{\ell m}(\mathbf{n})=Y_{\ell m}(\mathbf{n})$. There exist different conventions concerning these functions, see a survey in \cite{MR2884225}. In what follows, we use conventions by \cite{Leistedt_2012}.

Let $(Q+\mathrm{i}U)(\theta,\varphi)$ be an isotropic random field in the \emph{spin bundle} $(\mathrm{SO}(3)\times_{\mathrm{SO}(2)}H_s,p,S^2)$ with $s=2$. Then it has the form
\begin{equation}\label{eq:9}
(Q\pm\mathrm{i}U)(\mathbf{n})=\sum_{\ell=2}^{\infty}\sum_{m=-\ell}^{\ell}
a^{(\pm 2)}_{\ell m}\,{}_{\pm 2}Y_{\ell m}(\mathbf{n}),
\end{equation}
where
\[
a_{\pm 2,\ell m}=\int_{\mathbb{C}P^1}(Q\pm\mathrm{i}U)(\mathbf{n})\,
\overline{{}_{\pm 2}Y_{\ell m}(\mathbf{n})}\,\mathrm{d}\mathbf{n}.
\]
The real-valued random fields $Q(\mathbf{n})$ and $U(\mathbf{n})$ describe the Stokes parameters of the linear polarisation of the Cosmic Microwave Background. The expansion~\eqref{eq:9} appeared in \cite{PhysRevD.55.1830}.

Define
\[
e_{\ell m}=\RE a^{(2)}_{\ell m},\qquad b_{\ell m}=\IM a^{(2)}_{\ell m}.
\]
Under the \emph{parity transformation} $\mathbf{n}\mapsto-\mathbf{n}$, the spin $s$ spherical harmonics are transformed as
\[
{}_sY_{\ell m}(-\mathbf{n})=(-1)^{\ell}\,{}_{-s}Y_{\ell m}(\mathbf{n}).
\]
Then we have
\[
(Q\pm\mathrm{i}U)(-\mathbf{n})=\sum_{\ell=2}^{\infty}\sum_{m=-\ell}^{\ell}
(e_{\ell m}\pm\mathrm{i}b_{\ell m})(-1)^{\ell}\,{}_{\mp 2}Y_{\ell m}(\mathbf{n}).
\]
It follows that under the parity transformation, $e_{\ell m}$ remains invariant, while $b_{\ell m}$ changes sign. Following \cite{PhysRevD.55.1830}, introduce the random fields
\[
E(\mathbf{n})=\sum_{\ell=2}^{\infty}\sum_{m=-\ell}^{\ell}e_{\ell m}Y_{\ell m}(\mathbf{n}),\qquad B(\mathbf{n})=\sum_{\ell=2}^{\infty}\sum_{m=-\ell}^{\ell}b_{\ell m}Y_{\ell m}(\mathbf{n}).
\]
The random field $E(\mathbf{n})$ is scalar, like the electric field, while $B(\mathbf{n})$ is \emph{pseudo-scalar}, like the magnetic field, hence the notation.

\end{example}

\begin{example}[A tensor bundle]\label{ex:3}

Put $\mathbb{K}=\mathbb{R}$, $G=\mathrm{O}(3)$, $K=\mathrm{O}(2)$. Let $E$ be the real linear space of Hermitian linear operators over a two-dimensional complex linear space with inner product $(A,B)=\tr AB$, where the action of $K$ is given by $k\cdot A=kAk^{-1}$, $k\in K$, $A\in E$.

The group $\mathrm{O}(2)$ has the following irreducible orthogonal representations: $E^{0+}(k)=1$, $E^{0-}(k)=\det k$, and the representations $E^m$, $m\geq 1$, acting by
\begin{equation}\label{eq:12}
\begin{aligned}
E^m
\begin{pmatrix}
  \cos\varphi & -\sin\varphi \\
  \sin\varphi & \cos\varphi
\end{pmatrix}
&=
\begin{pmatrix}
  \cos(m\varphi) & -\sin(m\varphi) \\
  \sin(m\varphi) & \cos(m\varphi)
\end{pmatrix}
,\\
E^m
\begin{pmatrix}
  \cos\varphi & \sin\varphi \\
  \sin\varphi & -\cos\varphi
\end{pmatrix}
&=
\begin{pmatrix}
  \cos(m\varphi) & \sin(m\varphi) \\
  \sin(m\varphi) & -\cos(m\varphi)
\end{pmatrix}
.
\end{aligned}
\end{equation}

In contrast to the previous examples, this time the representation $E$ is
\emph{reducible} and contains three inequivalent irreducible components. The first component is $E^{0+}$. It acts in the one-dimensional subspace $E_+$ generated by the matrix
\[
\frac{1}{\sqrt{2}}\sigma_0=\frac{1}{\sqrt{2}}
\begin{pmatrix}
  1 & 0 \\
  0 & 1
\end{pmatrix}
.
\]
The elements of this space are \emph{scalars}. The second component is $E^{0-}$. It acts in the one-dimensional subspace $E_-$ generated by the matrix
\[
\frac{1}{\sqrt{2}}\sigma_2=\frac{1}{\sqrt{2}}
\begin{pmatrix}
  0 & -\mathrm{i} \\
  \mathrm{i} & 0
\end{pmatrix}
.
\]
The elements of this subspace are \emph{pseudo-scalars}. Finally, the third irreducible component is $E^2$. It acts in the two-dimensional space $E_2$ of symmetric trace-free matrices generated by
\[
\frac{1}{\sqrt{2}}\sigma_1=\frac{1}{\sqrt{2}}
\begin{pmatrix}
  0 & 1 \\
  1 & 0
\end{pmatrix}
,\qquad \frac{1}{\sqrt{2}}\sigma_3=\frac{1}{\sqrt{2}}
\begin{pmatrix}
  1 & 0 \\
  0 & -1
\end{pmatrix}
.
\]
The matrices $\sigma_i$ are known as \emph{Pauli matrices}.

We use Frobenius reciprocity to analyse the structure of the space $L^2(S^2,\mathrm{O}(3)\times_{\mathrm{O}(2)}E)$. Let $V$ be an irreducible orthogonal representation of the group~$\mathrm{O}(3)$. The group $\mathrm{O}(3)$ is the Cartesian product of its subgroups $\mathrm{SO}(3)$ and $Z^c_2=\{\pm e\}$. Therefore, $V=V^{\ell\pm}=V^{\ell}\otimes V^{\pm}$, where $V^{\ell}$ is the irreducible orthogonal representation of the group $\mathrm{SO}(3)$, $V^+$ is the trivial representation of the group $Z^c_2$, and $V^-$ is the representation $g\mapsto\det g$ of $Z^c_2$.

\begin{lemma}\label{lem:2}
The restriction of the representation $V^{\ell\pm}$ to the subgroup $\mathrm{O}(2)$ has the form
\begin{equation}\label{eq:10}
\begin{aligned}
\res^G_KV^{2\ell\pm}&=E^{0\pm}\oplus E^1\oplus\cdots\oplus E^{2\ell},\\
\res^G_KV^{(2\ell+1)\pm}&=E^{0\mp}\oplus E^1\oplus\cdots\oplus E^{2\ell+1}.
\end{aligned}
\end{equation}
\end{lemma}

\begin{proof}

Note that the group~$\mathrm{O}(2)$ is embedded into $\mathrm{O}(3)$ as follows:
\[
\mathrm{O}(2)\ni k\mapsto
\begin{pmatrix}
  k & \mathbf{0} \\
  \mathbf{0}^{\top} & 1
\end{pmatrix}
\in\mathrm{O}(3).
\]
It becomes obvious that \eqref{eq:10} is true when $\ell=0$. Assume that Lemma~\ref{lem:2} is proved for all values of indices up to $\ell$. Then we have
\[
\begin{aligned}
\res^G_KV^{(2\ell+1)+}&=E^{0-}\oplus E^1\oplus\cdots\oplus E^{2\ell+1},\\
\res^G_KV^{1+}&=E^{0-}\oplus E^1.
\end{aligned}
\]
It follows that
\[
\res^G_K(V^{(2\ell+1)+}\otimes V^{1+})|=(E^{0-}\oplus E^1\oplus\cdots\oplus E^{2\ell+1})\otimes(E^{0-}\oplus E^1).
\]
The left hand side of this equality is
\[
\begin{aligned}
&\res^G_KV^{2\ell+}\oplus\res^G_KV^{(2\ell+1)+}\oplus \res^G_KV^{(2\ell+2)+}=E^{0+}\oplus E^1\oplus\cdots\oplus E^{2\ell}\\
&\quad\oplus E^{0-}\oplus E^1\oplus\cdots\oplus E^{2\ell+1}\oplus\res^G_K V^{(2\ell+2)+}
\end{aligned}
\]
by the induction hypothesis. The right hand side is
\[
E^{0+}\oplus E^1\oplus\cdots\oplus E^{2\ell+1}\oplus E^1\oplus E^{0+}\oplus E^{0-}\oplus E^2\oplus\cdots\oplus E^{2\ell}\oplus E^{2\ell+2}.
\]
Here we used the fact that $E^{0-}\otimes E^{0-}=E^{0+}$, $E^m\otimes E^{0-}=E^m$, $E^1\otimes E^1=E^{0+}\oplus E^{0-}\oplus E^2$, and $E^1\otimes E^m=E^{m-1}\oplus E^{m+1}$ for $m\geq 2$. To prove this, recall that the \emph{character} of a finite-dimensional representation is the trace of its matrix. The character is independent on the choice of a basis. Moreover, the character of the tensor product of finitely many representations is equal to the product of the characters of the terms, while that of the direct sum of finitely many representations is equal to the some of the characters of the terms. Using equation~\eqref{eq:12}, we prove that the characters of both hand sides are equal. It follows from the last two displays that
\[
\res^G_KV^{(2\ell+2)+}=E^{0+}\oplus E^1\oplus\cdots\oplus E^{2\ell+2}.
\]
Multiply this equality by the equality $\res^G_KV^{1-}=E^{0-}\oplus E^1$. We obtain
\[
\res^G_KV^{(2\ell+2)-}=E^{0-}\oplus E^1\oplus\cdots\oplus E^{2\ell+2}.
\]
The induction step from $2\ell+1$ to $2\ell+2$ is proved. The step from $2\ell+2$ to $2\ell+3$ is doing similarly.

\end{proof}

We have $E=E_+\oplus E_-\oplus E_2$. It follows that
\[
\mathrm{O}(3)\times_{\mathrm{O}(2)}E=(\mathrm{O}(3)\times_{\mathrm{O}(2)}E_+)
\oplus(\mathrm{O}(3)\times_{\mathrm{O}(2)}E_-)\oplus(\mathrm{O}(3)\times_{\mathrm{O}(2)}E_2),
\]
and similar equality is true for the spaces of the square-integrable sections of the above bundles:
\[
H=H_+\oplus H_-\oplus H_2.
\]
By Lemma~\ref{lem:2}, the representation~$E^{0+}$ belongs to the restrictions to $\mathrm{O}(2)$ of the representations $V^{0+}$, $V^{1-}$, $V^{2+}$, \dots. By Frobenius reciprocity,
\[
n(V^{\ell(-1)^{\ell}},H_+)=\frac{n(E^{0+},
\res^G_KV^{\ell(-1)^{\ell}})}{\dim_{\mathbb{R}}D(V^{\ell(-1)^{\ell}})}.
\]
To calculate $\dim_{\mathbb{R}}D(V)$ for a real irreducible representation $V$, do the following. Define the \emph{extended representation} by $e^{\mathbb{C}}_{\mathbb{R}}V=\mathbb{C}\otimes_\mathbb{R}V$. By \cite[Chapter~2, Proposition~6.6, Theorem~6.7]{MR1410059}, there may be three cases:

\begin{enumerate}
  \item $e^{\mathbb{C}}_{\mathbb{R}}V$ is irreducible. In this case, $D(V)=\mathbb{R}$ and $\dim_{\mathbb{R}}D(V)=1$.
  \item $e^{\mathbb{C}}_{\mathbb{R}}V$ is a direct sum of two non-equivalent components. In this case, $D(V)=\mathbb{C}$ and $\dim_{\mathbb{R}}D(V)=2$.
  \item $e^{\mathbb{C}}_{\mathbb{R}}V$ is a direct sum of two equivalent components. In this case, $D(V)=\mathbb{H}$ and $\dim_{\mathbb{R}}D(V)=4$.
\end{enumerate}

It is easy to check that all the real irreducible representations of both $G$ and $K$ belong to the first class. In other words, Frobenius reciprocity has the usual form. We have
\[
H_+=\sum_{\ell=0}^{\infty}\oplus V^{\ell(-1)^{\ell}},\qquad H_-=\sum_{\ell=0}^{\infty}\oplus V^{\ell(-1)^{\ell+1}}
\]
and
\begin{equation}\label{eq:14}
H_2=\sum_{\ell=2}^{\infty}\oplus (V^{\ell+}\oplus V^{\ell-}).
\end{equation}

An isotropic random field in the bundle $(\mathrm{O}(3)\times_{\mathrm{O}(2)}E_+,p,S^2)$ takes the form
\[
I(\mathbf{n})=\sum_{\ell=0}^{\infty}\sum_{m=-\ell}^{\ell}a^I_{\ell m}Y_{\ell m}(\mathbf{n}),
\]
and describes the first Stokes parameter, the intensity $I(\mathbf{n})$ of the Cosmic Microwave Background, which is proportional to the fourth power of its absolute temperature, $T(\mathbf{n})$, by the \emph{Stephan--Boltzmann law}. Like $E(\mathbf{n})$, this is a scalar field, a section of the vector bundle generated by the trivial representation of the subgroup $\mathrm{O}(2)$.

For the bundle $(\mathrm{O}(3)\times_{\mathrm{O}(2)}E_-,p,S^2)$ we have
\[
V(\mathbf{n})=\sum_{\ell=0}^{\infty}\sum_{m=-\ell}^{\ell}a^V_{\ell m}Y_{\ell m}(\mathbf{n}).
\]
This field describes the \emph{circular polarisation} $V(\mathbf{n})$ of the Cosmic Microwave Background. Like $B(\mathbf{n})$, this is a pseudo-scalar field, an isotropic random section of the vector bundle generated by the representation $g\mapsto\det g$ of the subgroup $\mathrm{O}(2)$.

Finally, an isotropic random section of the bundle $(\mathrm{O}(3)\times_{\mathrm{O}(2)}E_2,p,S^2)$ has the form
\[
\begin{pmatrix}
  Q(\mathbf{n}) & U(\mathbf{n}) \\
  U(\mathbf{n}) & -Q(\mathbf{n})
\end{pmatrix}
=\sum_{\ell=2}^{\infty}\sum_{m=-\ell}^{\ell}[a^G_{\ell m}Y^G_{\ell m}(\mathbf{n})+a^C_{\ell m}Y^C_{\ell m}(\mathbf{n})],
\]
where we denote by $\{\,Y^G_{\ell m}(\mathbf{n})\colon\ell\geq 2,-\ell\leq m\leq\ell\,\}$ a basis in the space $V^{\ell(-1)^{\ell}}$ of the expansion~\eqref{eq:14}, while $\{\,Y^C_{\ell m}(\mathbf{n})\colon\ell\geq 2,-\ell\leq m\leq\ell\,\}$ is a basis in the space $V^{\ell(-1)^{\ell+1}}$ of the above expansion. Comparing our expansion with \cite[Equation~2.10]{Kamionkowski:1996ks}, we identify the introduced basis functions with \emph{tensor spherical harmonics} described there. See also a survey \cite{MR569166} and \cite{MR0270692}.

\end{example}

\begin{example}

The Universe is identified with a four-dimensional manifold, say $M$. If the manifold $M$ were unperturbed, we would define a preferred observer as that who sees zero momentum density at his/her own position, see \cite{MR1754145}. Such observers are called \emph{comoving}. Let $(t,\mathbf{x})$ be a chart centred at a comoving observer. Then, the spacetime $M$ \emph{threads into lines} corresponding to fixed $\mathbf{x}$ and \emph{slices into hypersurfaces} corresponding to fixed $t$. The slicing is orthogonal to the threading, and, on each slice, the Universe is homogeneous.

In the presence of perturbations, it is impossible to find a chart satisfying all the above properties. A chart is called a \emph{gauge} if it converges to the comoving chart in the limit where the perturbations vanish. In what follows, we use the so called \emph{total-matter gauge} $(t,r,\mathbf{n})$ described in \cite[Subsection~14.6.4]{MR1754145}.

Let $\delta(r,\mathbf{n})$ be the overdensity field:
\[
\delta(r,\mathbf{n})=\frac{\rho(r,\mathbf{n})-\bar{\rho}}{\bar{\rho}},
\]
where $\rho(r,\mathbf{n})$ is the matter density at a point $(r,\mathbf{n})$, and $\bar{\rho}$ is the average matter density. The overdensity field is not observable because most part of matter is dark and invisible. To overcome this difficulty, we proceed in three steps.

\begin{enumerate}
  \item Relate the overdensity field $\delta(r,\mathbf{n})$ and the Newtonian potential $\Phi(r,\mathbf{n})$ by Poisson's equation
      \[
      \nabla^2\Phi(r,\mathbf{n})=\frac{3\Omega_MH^2_0}{2a(r)}\delta(r,\mathbf{n}),
      \]
      where $\Omega_M$ is the dimensionless matter density, $H_0$ is the current value of the Hubble parameter, and $a(r)$ is the dimensionless scale factor.
  \item Define the lensing potential by
  \[
  \phi(r,\mathbf{n})=\frac{2}{c^2}\int_{0}^{r}\Phi(r',\mathbf{n})\frac{r-r'}{rr'}
  \,\mathrm{d}r',
  \]
  where $c$ is the speed of light in a vacuum.
  \item For a fixed $s\in\mathbb{Z}$, define a differential operator $\eth$ (in fact, a family of operators) by
      \[
      \eth=s\cot\theta-\frac{\partial}{\partial\theta}
      -\frac{\mathrm{i}}{\sin\theta}\frac{\partial}{\partial\varphi}.
      \]
      Then we have
      \[
      \eth{}_sY_{\ell m}(\mathbf{n})=
      \begin{cases}
        0, & \mbox{if } \ell=s\geq 0, \\
        \sqrt{(\ell-s)(\ell+s+1)}{}_{s+1}Y_{\ell m}(\mathbf{n}), & \mbox{otherwise}.
      \end{cases}
      \]
      In other words, $\eth$ maps a section of the spin-$s$ bundle to a section of spin-$s+1$ bundle, hence the name \emph{spin-rising operator}.
      The conjugate operator
      \[
      \eth^*=s\cot\theta-\frac{\partial}{\partial\theta}
      +\frac{\mathrm{i}}{\sin\theta}\frac{\partial}{\partial\varphi}
      \]
      lowers the spin:
      \[
      \eth^*{}_sY_{\ell m}(\mathbf{n})=
      \begin{cases}
        0, & \mbox{if } \ell=-s\geq 0, \\
        \sqrt{(\ell+s)(\ell-s+1)}{}_{s-1}Y_{\ell m}(\mathbf{n}), & \mbox{otherwise},
      \end{cases}
      \]
      hence the name \emph{spin-lowering operator}.

      Define the \emph{distortion fields} ${}_sX((r,\mathbf{n})$ by
\begin{equation}\label{eq:13}
\begin{aligned}
{}_0\kappa(r,\mathbf{n})&=\frac{1}{4}(\eth\eth^*+\eth^*\eth)\phi(r,\mathbf{n}),\\
{}_1\mathcal{F}(r,\mathbf{n})&=-\frac{1}{6}(\eth^*\eth\eth+\eth\eth^*\eth+\eth\eth\eth^*)\phi(r,\mathbf{n}),\\
{}_2\gamma(r,\mathbf{n})&=\frac{1}{2}\eth^2\phi(r,\mathbf{n}),\\
{}_3\mathcal{G}(r,\mathbf{n})&=-\frac{1}{2}\eth^3\phi(r,\mathbf{n}).
\end{aligned}
\end{equation}
\end{enumerate}

The constructed fields are called the \emph{magnification}, the \emph{first flexion}, the \emph{shear}, and the \emph{third flexion}. In contrast to the overdensity field, they are \emph{observable}. To find their spectral expansions, use the idea formulated by M.\u{I}. Yadrenko in \cite{MR0164376}. Suppose these fields are mean-square continuous. The restriction of the field ${}_sX(r,\mathbf{n})$ to the centred sphere of a fixed radius $r_0>0$ is an isotropic random section of the bundle $(\mathrm{O}(3)\times_{\mathrm{O}(2)}E_s,p,S^2)$ and has the form
\[
{}_sX(r_0,\mathbf{n})=\sum_{\ell=s}^{\infty}\sum_{m=-\ell}^{\ell}{}_sa_{\ell m}(r_0)\,{}_sY_{\ell m}(\mathbf{n}),\qquad 0\leq s\leq 3.
\]
Varying $r_0$, we obtain
\[
{}_sX(r,\mathbf{n})=\sum_{\ell=s}^{\infty}\sum_{m=-\ell}^{\ell}{}_sa_{\ell m}(r)\,{}_sY_{\ell m}(\mathbf{n}),
\]
where ${}_sa_{\ell m}(r)$ is a sequence of stochastic processes satisfying $\mathsf{E}[{}_sa_{\ell m}(r)]=0$ unless $s=\ell=m=0$. For simplicity, put $\mathsf{E}[{}_0a_{00}(r)]=0$. Moreover, we have
\[
\mathsf{E}[{}_sa_{\ell m}(r_1)\overline{{}_sa_{\ell'm'}(r_2)}]=\delta_{\ell\ell'}
\delta_{mm'}\,{}_sC_{\ell}(r_1,r_2),
\]
with
\[
\sum_{\ell=s}^{\infty}(2\ell+1){}_sC_{\ell}(r,r)<\infty,\qquad r\geq 0.
\]

The stochastic processes ${}_sa_{\ell m}(r)$ are expressed through the values of the distortion field as
\[
{}_sa_{\ell m}(r)=\int_{S^2}{}_sX(r,\mathbf{n})\,\overline{{}_sY_{\ell m}(\mathbf{n})}\,\mathrm{d}\mathbf{n}
\]
and are usually represented by
\begin{equation}\label{eq:15}
{}_sa_{\ell m}(r)=\sqrt{\frac{2}{\pi}}\int_{0}^{\infty}j_{\ell}(kr){}_s\tilde{a}_{\ell m}(k)k^2\,\mathrm{d}k,
\end{equation}
where
\[
{}_s\tilde{a}_{\ell m}(k)=\sqrt{\frac{2}{\pi}}\int_{0}^{\infty}j_{\ell}(kr){}_sa_{\ell m}(r)r^2\,\mathrm{d}r
\]
is the Fourier--Bessel (or Hankel) transform of the process ${}_sa_{\ell m}(r)$, and $k$ is the radial wavenumber. Here we implicitly suppose that the sample paths of the above process are a.s. square-integrable. The spectral expansion of the distortion field takes the form
\[
{}_sX(r,\mathbf{n})=\sqrt{\frac{2}{\pi}}\sum_{\ell=s}^{\infty}\sum_{m=-\ell}^{\ell}
\int_{0}^{\infty}j_{\ell}(kr)\,{}_s\tilde{a}_{\ell m}(k)k^2\,\mathrm{d}k\,{}_sY_{\ell m}(\mathbf{n}),
\]
see \cite{Castro_2005,Heavens_2003,Kitching_2011,Munshi_2011}.

The drawback of the above approach is as follows: there exist no method to compute the transform \eqref{eq:15} exactly for a useful class of functions, see \cite{Leistedt_2012}. Moreover, the spherical Bessel functions $j_{\ell}(kr)$ are highly oscillating, and finding a good quadrature formula is a complicated issue. To overcome these difficulties, wavelet methods were proposed by \cite{Leistedt_2012,Leistedt_2015}. We propose a different approach.

Assume the following: the distortion field is Gaussian, has a.s. continuous sample paths and is observed in the ball of radius~$R$. Then, the stochastic processes ${}_sa_{\ell m}(r)$ are Gaussian, independent, a.s. continuous, and their distributions depend on $\ell$ but do not depend on $m$. Each process generates a Gaussian centred measure ${}_s\mu_{\ell}$ in the real separable Banach space $C[0,R]$. Let ${}_sH_{\ell}$ be the reproducing kernel Hilbert space of ${}_s\mu_{\ell}$ with inner product $\langle\boldsymbol{\cdot},\boldsymbol{\cdot}\rangle$, see \cite{MR1435288}. A sequence $\{\,{}_sf_{\ell j}(r)\colon j\geq 1\,\}$ of the elements of ${}_sH_{\ell}$ is called a \emph{Parceval frame} \cite{MR2511282} if
\[
\sum_{j=1}^{\infty}\langle{}_sf_{\ell j},h\rangle\,{}_sf_{\ell j}=h
\]
for all $h\in{}_sH_{\ell}$, where the series converges in the norm of ${}_sH_{\ell}$. By adding zeroes, finite sequences may also serve as frames. Let ${}_sX_{\ell j}$ be independent standard normal random variables. By \cite[Theorem~1]{MR2511282} the process ${}_sa_{\ell m}(r)$ has a.s. uniformly convergent expansion
\[
{}_sa_{\ell m}(r)=\sum_{j=1}^{\infty}{}_sf_{\ell j}(r)\,{}_sX_{\ell mj}
\]
if and only if the sequence $\{\,{}_sf_{\ell j}(r)\colon j\geq 1\,\}$ is a Parceval frame for ${}_sH_{\ell}$. The spectral expansion of the distortion field takes the form
\[
{}_sX(r,\mathbf{n})=\sum_{\ell=s}^{\infty}\sum_{m=-\ell}^{\ell}\sum_{j=1}^{\infty}{}_sf_{\ell j}(r)\,{}_sX_{\ell mj}\,{}_sY_{\ell m}(\mathbf{n}).
\]

In other words, each ``component''
\[
{}_sX_{\ell}(r,\mathbf{n})=\sum_{m=-\ell}^{\ell}{}_sa_{\ell m}(r)\,{}_sY_{\ell m}(\mathbf{n})
\]
of the distortion random field ${}_sX(r,\mathbf{n})$ is expanded with respect to a \emph{separable} Parceval frame
\[
\{\,{}_sf_{\ell j}(r)\,{}_sY_{\ell m}(\mathbf{n})\colon j\geq 1,-\ell\leq m\leq\ell\,\}
\]
for the Hilbert space ${}_sH_{\ell}\otimes L^2(S^2,\mathrm{O}(3)\times_{\mathrm{O}(2)}E_s)$, that is, the coordinates $r$ and $\mathbf{n}$ are separated. For each $\ell\geq s$, one can choose a suitable basis $\{\,{}_sf_{\ell j}(r)\colon j\geq 1\,\}$ in the Hilbert space $L^2([0,R],r^2\,\mathrm{d}r)$. These may be wavelets, orthogonal polynomials, etc. Assume that there are real numbers ${}_sc_{\ell j}$ such that the sequence $\{\,{}_sc_{\ell j}\,{}_sf_{\ell j}(r)\colon j\geq 1\,\}$ is a Parceval frame for ${}_sH_{\ell}$. The distortion random field takes the form
\[
{}_sX(r,\mathbf{n})=\sum_{\ell=s}^{\infty}\sum_{m=-\ell}^{\ell}\sum_{j=1}^{\infty}
{}_sc_{\ell j}\,{}_sf_{\ell j}(r)\,{}_sX_{\ell mj}\,{}_sY_{\ell m}(\mathbf{n}).
\]
The obtained expansion may serve as a base for the further statistical analysis of the distortion random field.

\end{example}

\end{document}